\documentclass{article}

\usepackage{arxiv}

\usepackage[utf8]{inputenc} 
\usepackage[T1]{fontenc}    
\usepackage{hyperref}       
\usepackage{url}            
\usepackage{booktabs}       
\usepackage{amsfonts}       
\usepackage{nicefrac}       
\usepackage{microtype}      
\usepackage{lipsum}
\usepackage{graphicx}
\graphicspath{ {./images/} }

\usepackage{amssymb}
\usepackage{color}
\usepackage{amsmath}
\usepackage{bm}
\usepackage{amsthm}
\usepackage{booktabs}
\usepackage{array}
\usepackage{empheq}
\newtheorem{theorem}{Theorem}
\newtheorem{lemma}{Lemma}
\newtheorem{remark}{Remark}
\newtheorem{example}{Example}

\title{Fourth-order compact difference schemes for the one-dimensional Euler-Bernoulli beam equation with damping term}

\author{
 Wenjie Huang \\
  School of Mathematics\\
  Sichuan University\\
  Chengdu, China, 610065 \\
  \texttt{jackson11235813@163.com} \\
   \And
 Hao Wang \\
  School of Mathematics\\
  Sichuan University\\
  Chengdu, China, 610065 \\
  \texttt{wangh@scu.edu.cn} \\ 
  \And
 Shiquan Zhang \\
  School of Mathematics\\
  Sichuan University\\
  Chengdu, China, 610065 \\
  \texttt{shiquanzhang@scu.edu.cn} \\  
  \And
 Qinyi Zhang \\
 School of Mathematics\\
  Sichuan University\\
  Chengdu, China, 610065 \\
  \texttt{zqy\_edu@163.com} \\
}

\begin{document}
\maketitle
\begin{abstract}
This paper proposes and analyzes a finite difference method based on compact schemes for the Euler-Bernoulli beam equation with damping terms. The method achieves fourth-order accuracy in space and second-order accuracy in time, while requiring only three spatial grid points within a single compact stencil. Spatial discretization is carried out using a compact finite difference scheme, with a variable substitution technique employed to reduce the order of the equation and effectively handle the damping terms. For the temporal discretization, the Crank–Nicolson scheme is applied. The consistency, stability, and convergence of the proposed method are rigorously proved. Numerical experiments are presented to verify the theoretical results and demonstrate the accuracy and efficiency of the method.
\end{abstract}



\section{Introduction}
\label{Intro}

Euler-Bernoulli beam theory is an important theory in structural mechanics that describes the relationship between the bending moment and deflection of a beam under lateral load\cite{truesdell1960rational} and this type of structure is common in various industrial fields, including aviation, automobiles, railways, shipbuilding, and civil engineering\cite{han1999dynamics,li2020flexural}. 

The core of Euler-Bernoulli beam theory is the Euler-Bernoulli beam equation, which is a fourth-order parabolic partial differential equation and is usually difficult or impossible to obtain the analytical solution\cite{bauchau2009euler,haider2023exact}, so its numerical solution is crucial. Many people have tried to solve the Euler-Bernoulli beam equation in history and have given many numerical solutions. Ahmed provides a detailed revision of the Euler–Bernoulli and Timoshenko beam theories and their analytical and numerical applications in an appropriate and simplified manner\cite{ahmed2021euler}.

The numerical solution of the Euler-Bernoulli equation mainly includes the finite difference method and the finite element method. Fogang applied the finite difference method (FDM) to Euler-Bernoulli beams, using the fourth-order polynomial hypothesis (FOPH) and additional points to improve accuracy, and proposed a direct time integration method (DTIM) for forced vibration analysis\cite{fogang2021euler}. Twizell  invented a finite difference method for numerically solving fourth-order parabolic partial differential equations with one or two spatial variables\cite{twizell1983difference}. Atilla proposed the variational derivative method to solve the Euler-Bernoulli beam functional, eliminating the need for out-of-region points in finite difference methods, and verified its performance through numerical comparisons\cite{ozutok2010solution}. This method is derived from the multi-derivative method of second-order ordinary differential equations. Awrejcewicz studied the chaotic vibration of a flexible nonlinear Euler-Bernoulli beam under harmonic loads and various symmetric or asymmetric boundary conditions, and verified the reliability of the results obtained by the finite difference method and the finite element method\cite{awrejcewicz2011analysis}. Mohanty propose a high-order compact scheme with fourth-order spatial and second-order temporal accuracy for solving 2D fourth-order PDEs, avoiding ghost points and ensuring high accuracy and stability\cite{mohanty2017high}. Niiranen developed and analyzed Euler-Bernoulli beam models using strain gradient elasticity, compared different beam models, and validated their results with numerical simulations, focusing on size effects and stiffening behavior\cite{niiranen2019variational}. Pathak used a high-order compact finite difference scheme in MATLAB to approximate the solution of the Euler-Bernoulli beam equation for beam deflection analysis\cite{pathak2019high}.  Aouragh gave a numerical solution to the undamped Euler-Bernoulli beam equation based on a compact difference scheme\cite{aouragh2024compact}.

The finite element method is also used to solve the Euler-Bernoulli equation. Bardell introduced a hybrid h-p finite element technique for static analysis of Euler-Bernoulli beams\cite{bardell1996engineering}.  Wei use Hermite cubic finite elements to approximate the solutions of a nonlinear Euler-Bernoulli beam equation\cite{wei2012analytic}. Eltaher studied an effective finite element model for vibration analysis of nonlocal Euler-Bernoulli beams and proposed Eringen's nonlocal constitutive equations\cite{eltaher2013vibration}. Nguyen invented a hybrid finite element method for the static analysis of nanobeams and derived the governing equations of the Euler-Bernoulli beam theory by combining nonlocal theory\cite{nguyen2015mixed}.  Shang used the generalized finite element method (GFEM) to perform dynamic analysis on one-dimensional rod and Euler-Bernoulli beam problems and analyzed the free vibration problem of beams to evaluate the robustness and efficiency of elements\cite{shang2016dynamic}.Zakeri used the basic displacement function (BDF) to perform structural analysis of nanobeams using the finite element method based on Eringen nonlocal elasticity and Euler Bernoulli beam theory\cite{zakeri2016analysis}.  Lin presents an error analysis of a Hermite cubic immersed finite element (IFE) method for solving interface problems of the differential equation modeling a Euler-Bernoulli beam made up of multiple materials together with suitable jump conditions at material interfaces\cite{lin2017error}. The nonlocal Euler-Bernoulli motion equations were derived on the basis of variational statements. Wang et al. studied the displacement recovery technique through superconvergent patches and developed an adaptive finite element method for the structural characteristic problem of cracked Euler-Bernoulli beams\cite{wang2018adaptive}. ABRO presents a finite element numerical method for Euler-Bernoulli beams with variable coefficients\cite{abro2020numerical}.

Among the existing numerical methods for solving fourth-order parabolic partial differential equations, the finite difference method is favored because of its simple theoretical derivation, relatively easy implementation, flexibility and accurate numerical solution.  However, the damping term is not considered in many numerical methods for solving the Euler-Bernoulli beam equation. 

Therefore, the aim of our paper is to develop a numerical method to solve the Euler-Bernoulli beam equation with the damping term. This method employs a compact finite difference for the spatial derivative using a 3-point compact stencil and a second-order Crank–Nicolson approximation for the temporal derivative. Our proof to stability is general and can provide guidance to prove the stability and consistence for similar difference schemes.

The outline of this paper is as follows. In the second section, we present the process of decomposing \eqref{ebe} into a system of two second-order equations. Furthermore, we approximate the second-order spatial derivative using a fourth-order compact finite difference scheme, from which we obtain a system of ordinary differential equations. This system is solved using the Crank–Nicolson scheme. In the third section, we discuss the convergence, stability and convergence of the proposed scheme by the matrix method, which is shown to be unconditionally stable and second-order accurate in time and fourth-order accurate in space. In the fourth section, we present numerical results for three examples and compare errors with those from existing works in the literature, confirming the order of convergence. Finally, in the last section, we mention the conclusion of the implementation and the superiority of the proposed method based on numerical results from the problems.

\section{The Damping Model And Its Compact Difference scheme}
\subsection{Euler-Bernoulli beam mode with damping term}

We consider a one-dimensional Euler-Bernoulli beam model with damping in the domain $Q_T =[0, L]\times [0, T]$, where $L$ is the beam length and $T$ represents the \textit{total time}. The problem with the model in the domain is to find $u=u(x,t)$ that satisfies the following equation with boundary conditions:

\begin{empheq}[left=\empheqlbrace]{align}
     &\frac{\partial^2 }{\partial x^2}\left(E(x)I(x)\frac{\partial^2 u(x,t)}{\partial x^2}\right) + \rho(x)\frac{\partial^2u(x,t)}{\partial t^2} + c(x)\frac{\partial u(x,t)}{\partial t} = f(x,t),\ (x,t) \in Q_T,\label{ebe}\\
     &u(x,0) = \xi_1(x), u_t(x,0) = \xi_2(x),\label{ebe-initial-condition}\\
     &u(0,t)  = \mu_0(t),u(L,t)  = \mu_1(t), u_{xx}(0,t)= \mu_2(t),u_{xx}(L,t)  = \mu_3(t)\label{ebe-boundary-condition}.
\end{empheq}

where $u_t:=\frac{\partial u}{\partial t}$ and $u_{xx}:= \frac{\partial^2 u}{\partial x^2}$. \eqref{ebe-initial-condition} represents the initial condition and \eqref{ebe-boundary-condition} represents the boundary condition with the assumption that the beam is simply supported (or hinged) at its ends. In equation \eqref{ebe}, $u$ represents the displacement of a point on the beam at position $x$ and time $t$ relative to the equilibrium position,  $E$ represents the elastic modulus of the material, $I$ represents the section inertia moment, $\rho$ represents the density of material,    $c$ represents the damping coefficient, and $f$ represents the external force . In this paper, we only consider the case where $E$, $I$, $\rho$, and $c$ are all constants.

According to the first and second terms of equation \eqref{ebe}, the solution $u$ needs to satisfy $u\in C^4$ in the $x$ direction and $u\in C^2$ in the $t$ direction, and therefore we define the space 
\begin{equation}
    \mathcal{U} = \left\{u(x,t) \in Q(x,t) \Big|\frac{\partial^4 u}{\partial x^4} \text{and} \frac{\partial^2 u}{\partial t^2}\text{exist},u(x,t) \ \text{satisfy the equation \eqref{ebe}, \eqref{ebe-initial-condition} and \eqref{ebe-boundary-condition}} \right\}
\end{equation}
The force term $f=f(x,t)$ is generally assumed to satisfy $f\in C^0$ since the difference schemes that this paper focuses on are required. The third term of equation \eqref{ebe} is also called \textit{damping term} because it causes the dissipation of energy.

\subsection{Compact difference method}
\label{discrete}
We consider discretizing the problem onto a uniform grid. Let $N_x, N_T$ be a positive integer, we uniformly divide $\Omega$ into $N_x$ parts in $x$ direction and $N_T$ parts in $t$ direction, denote the step size of spatial variable $x$ by $h = \Delta x = \frac{L}{N_x}$ and the step size of time variable $t$ by $\Delta t = \frac{T}{N_T}$, and write $(\Delta x)^y, (\Delta t)^y$ as $\Delta x^y, \Delta t^y$. We denote the corresponding admissible set of spatial coordinates by $\mathbf{X} :=\left\{ 0 , 1,...,N_x\right\}$ and time coordinates by $\mathbf{T}:= \left\{ 0 , 1,...,N_T\right\}$.
%
%
%
%
For a function $v\in \mathcal{U}$, we define the finite difference operators 
\begin{equation}
        \delta^c_kv^n_i:=\sum^{k}_{j=-k}\hat{b}_jv_{i+j},
\end{equation}
where $\hat{b}_j$ is fixed corresponding to center difference scheme when $N,c$ are fixed, and $v^n_i=v(i\Delta x, n\Delta t), i\in \mathbf{X}, n\in\mathbf{T}$. Again, we define compact difference operators of $c$-order derivative
\begin{align}
    \label{compact_operator}
    D^c_N v^n_i:= \sum_{j=-N}^{N}a_j \left( \frac{\partial^c v}{\partial x^c} \right)^{n}_{i+j},\ a_j>0 \ ,\ N \leq i \leq N_x-N.
\end{align}
where $a_j$ represents the weight of the term in compact scheme and notice that the number of points we used is $2N+1 $ by definition \eqref{compact_operator}.
Now the compact difference scheme of $c$-order derivative is able to be expressed in the form of
\begin{equation}
    \label{compactscheme}
    D^c_N v^{n}_i= \frac{1}{\Delta x^c}\sum_{j=1}^{\hat{N}}b_j \delta^c_j v^{n}_i,
\end{equation}
the integer $N, \hat{N}$ is defined to satisfies the consistence
\begin{equation}
    \label{compact_consistence}
     D^c_N v^{n}_i - \frac{1}{\Delta x^c}\sum_{j=1}^{\hat{N}}b_j \delta^c_j v^{n}_i = O(\Delta x^k).
\end{equation}

where $k$ is fixed and \eqref{compactscheme}-\eqref{compact_consistence} mean that we use finite difference scheme to approximate the weighted average of the derivatives of multiple points instead of the derivative of one point. The advantage of this is that a higher order of convergence can be obtained with fewer discrete points, especially in some cases of uniform grid. We use Taylor series to deduce the the integer $N, \hat{N}$ with a fixed appropriate integer $k$. 

For convenience, we also define a operator $\tilde{D}_N^c $:
\begin{equation}
    \tilde{D}_N^c  v_i^n := \sum_{j=-N}^{N}a_j v^{n}_{i+j},\ a_j>0 \ ,\ N \leq i \leq N_x-N.
\end{equation}
where $a_j$ is the coefficient of $D^c_Nv^n_i$ defined by \eqref{compact_operator}. That means we ignore the differential effect when we use the operator $ \tilde{D}_N^c$ which corresponds to operator $D_N^c$.

In this paper, the scheme we used of second derivative with $k=4$ is
\begin{equation}
    \label{D23}
    D^2_1v^{n}_i = \frac{1}{\Delta x^2}\delta^2_1 v^{n}_i,
\end{equation}
here we choose $N=1, \hat{N}=1$ for $k=4$ and expand \eqref{D23} in uniform grid with $b_1 = 1$ and $a_1 = a_{-1} = 1/12, a_0 = 5/6$
\begin{equation}
    \label{parameter}
    \frac{1}{12} \left (\frac{\partial^2 v}{\partial x^2}  \right ) ^n_{i+1}+ \frac{5}{6}\left (\frac{\partial^2 v}{\partial x^2}  \right ) ^n_{i}+\frac{1}{12} \left (\frac{\partial^2 v}{\partial x^2}  \right ) ^n_{i-1}=
    \frac{1}{\Delta x^2}\delta^2_1 v^{n}_i.
\end{equation}

With the scheme above,  we can discrete the equation and solve it.

\subsection{Compact difference scheme separating first and second time derivatives}
First, we replace the variables with the following form and transform the fourth-order equation into a second-order equation system,
\begin{empheq}[left=\empheqlbrace]{equation}
\begin{aligned}
    \phi(x,t) &=  u_t(x,t),\\
    \psi(x,t) &= u_{xx}(x,t).
\end{aligned}
\label{variable_assume}
\end{empheq}
Then equation \eqref{ebe} turns to
\begin{empheq}[left=\empheqlbrace]{align}
    EI\psi_{xx}(x,t) &=  -\rho\phi_t(x,t) - c\phi(x,t) - f(x,t), \label{eq1}\\ 
    \phi_{xx}(x,t) &= \psi_t(x,t). \label{eq2}   
\end{empheq}
In the variable substitution, $\phi(x,t)$ represents the velocity of the beam, and $\psi(x,t)$ represents the bending moment value of the beam. For simplicity in analysis, we rewrite the equation \eqref{eq1}-\eqref{eq2} with operators in the following form:
\begin{equation}
    \label{lu_f}
    \mathcal{L}H=\mathcal{F},
\end{equation}
where $H=H\left(\phi, \psi \right)$ and $\mathcal{F}$ represents the force term. Exactly $\mathcal{L}$ is a vector operator with two components
\begin{align}
    \mathcal{L} = \begin{pmatrix}
        \mathcal{\bar{L}}\\
        \mathcal{\bar{\bar{L}}}
    \end{pmatrix},
\end{align}
where
\begin{align}
    \mathcal{\bar{L}}(\phi ,\psi) &= -EI\psi_{xx} -\rho\phi_t - c\phi,\\
    \mathcal{\bar{\bar{L}}}(\phi ,\psi) &=  \phi_{xx} - \psi_t.
\end{align}
To get an spatial semidiscrete scheme we apply compact difference scheme \eqref{compactscheme} to \eqref{eq1}, notice that here $c = 2$
\begin{align}
    \label{eq1_cs1}
    D^2_1 \psi^{n}_i = \frac{1}{\Delta x^2} \delta^2_1 \psi^{n}_i,
\end{align}
by the definition of compact scheme
\begin{align}
D^2_1 EI\psi^{n}_i = \tilde{D}^2_1\left ( f-\rho\phi_t -c \phi \right ) ^{n}_{i},
    \label{eq1_cs2}
\end{align}
Where $a_j$ is already given by left terms of equation \eqref{parameter}. And apply the same way to \eqref{eq2} and notice that
\begin{align}
  D^2_1 \phi^{n}_i = \tilde{D}^2_1\left (\psi_t \right ) ^n_{i}.
    \label{eq2_cs1}
\end{align}
We deduce the compact scheme of equation system \eqref{eq1}-\eqref{eq2} by combining \eqref{eq1_cs1}-\eqref{eq1_cs2} and \eqref{eq2_cs1}
\begin{empheq}[left=\empheqlbrace]{align}
\tilde{D}^2_1 ( f-\rho\phi_t - c\phi ) ^{n}_{i} &= \frac{1}{\Delta x^2} \delta^2_1 EI\psi^{n}_i,\label{cs1}\\
        \tilde{D}^2_1 \left (\psi_t  \right )^{n}_i &= \frac{1}{\Delta x^2}\delta^2_1 \phi^{n}_i\label{cs2},
\end{empheq}
rewrite \eqref{cs1} and \eqref{cs2} as systems of ordinary differential equations and implies
\begin{empheq}[left=\empheqlbrace]{align}
    \bm{A}\left(\rho\bm{\Phi}_t + c\bm{\Phi} - \bm{f}\right) &= \bm{B}\bm{\Psi},\label{linearsystem1}\\
    \bm{A}\bm{\Psi}_t &= -\bm{B}\bm{\Phi},\label{linearsystem2}
\end{empheq}
where we define
\begin{align}
    \bm{\Phi} &:= \left( \phi_1,\phi_2,...,\phi_{N_x-1}\right)^T,\bm{\Psi} := \left( \psi_1,\psi_2,...,\psi_{N_x-1} \right)^T,\bm{f} := \left( f_1,f_2,...,f_{N_x-1} \right)^T,\\
    \bm{\Phi}_t &:=\frac{\partial }{\partial t} \bm{\Phi} = \left(\left ( \frac{\partial \phi}{\partial t} \right ) _1,\left ( \frac{\partial \phi}{\partial t} \right ) _2,...,\left ( \frac{\partial \phi}{\partial t} \right ) _{N_x-1} \right)^T,\bm{\Psi}_t := \frac{\partial }{\partial t} \bm{\Psi} =\left( \left ( \frac{\partial \psi}{\partial t} \right ) _1,\left ( \frac{\partial \psi}{\partial t} \right ) _2,...,\left ( \frac{\partial \psi}{\partial t} \right ) _{N_x-1} \right)^T,
\end{align}
where  $\phi_i(t) = \phi(i\Delta x,t)$,  the coefficient matrix $\bm A, \bm B$ is given by the following equation
\begin{align}
    \bm A &:= \mathrm{diag}_{N_x -1}\left( \frac{1}{12},\frac{5}{6},\frac{1}{12} \right),\\
    \bm B &:= \frac{1}{\Delta x^2}\mathrm{diag}_{N_x -1}\left(-1,2,-1  \right).
\end{align}

the $\mathrm{diag}_{N_x -1}(a_{-1},a_0,a_1)$ is redefined here as tridiagonal matrix of $(N_x -1)\times(N_x -1)$ size:
\begin{equation}
    \mathrm{diag}_{N_x -1}\left( a_{-1}, a_0, a_1 \right) := \begin{pmatrix}
        a_0    & a_1  &      &     &  & & &  \\
        a_{-1} & a_0  & a_1  &    &  & & &    \\
          &a_{-1} & a_0 & a_1 &  & & &   \\
                      &   &     & ... & & & &   \\
                & &  & a_{-1}  & a_0& a_1 &  \\
                & & &   & a_{-1}& a_0 &a_1  \\
                & & & &  & a_{-1}  & a_0
    \end{pmatrix}_{(N_x -1)\times(N_x -1)}
\end{equation}
Notice that \eqref{linearsystem1}-\eqref{linearsystem2} can be organized into one linear system by combining $\bm{\Phi}$ and $\bm{\Psi}$ into a new vector,i.e.
\begin{align}
    \label{linearsystem}
    \mathcal{A}\bm{U}_t = \mathcal{B}\bm{U}+\bm{F},
\end{align}
where
\begin{align}
    \bm{U} := \begin{pmatrix}        \bm \Phi \\        \bm \Psi    \end{pmatrix},
    \bm{U}_t := \frac{\partial}{\partial t}\bm{U}
    =\begin{pmatrix}      \bm \Phi_t\\        \bm \Psi_t    \end{pmatrix},
    \bm{F} :=  \begin{pmatrix}\bm {F}_1 \\ \bm{F}_2   \end{pmatrix},
\end{align}
\begin{align}
\label{eq:the definition of matrix A and B}
    \mathcal{A} &:= \begin{pmatrix}
        \rho\bm A & \mathbf{0}\\
        \mathbf{0} & \bm A
    \end{pmatrix},
    \mathcal{B} := \begin{pmatrix}
        -c\bm A & EI\bm B\\
        -\bm B & \mathbf{0}
    \end{pmatrix},
\end{align}
and
\begin{align*}
    \bm{F}_1 &= \left( \alpha + \tilde{D}^2_1f_1,\tilde{D}^2_1f_2,\dots,\tilde{D}^2_1f_{N_x-2},\tilde{D}^2_1f_{N_x-1}+\beta \right)^T,\\
    \bm{F}_2 &= \left( \alpha',0,\dots,0,\beta' \right)^T, \\
\end{align*}
where $\alpha,\beta,\alpha',\beta'$ represents the boundary condition:
\begin{align*}
    \alpha & = - EI \frac{\psi_0(t)}{h^2}  - \rho\frac{\phi_0'(t)}{12}  - c\frac{\phi_0(t)}{12},  \\
    \beta &= -EI\frac{\psi_{N_x}(t)}{h^2}  - \rho\frac{\phi_{N_x}'(t)}{12}  - c\frac{\phi_{N_x}(t)}{12}, \\
    \alpha' &= \frac{\phi_0(t)}{h^2} - \frac{\psi_0'(t)}{12}, \\
    \beta' &= \frac{\phi_{N_x}(t)}{h^2} - \frac{\psi_{N_x}'(t)}{12} .
\end{align*}
In order to maintain the stability in time, we use Crank-Nicolson scheme\cite{sun2006efficient}, which yields
\begin{align}
    \label{The final Scheme}
    \mathcal{A}\left( \frac{\bm{U}^{n+1}-\bm{U}^{n}}{\Delta t} \right) = \mathcal{B}\left( \frac{\bm{U}^{n+1}+\bm{U}^{n}}{2} \right) + 
    \left( \frac{\bm{F}^{n+1}+\bm{F}^{n}}{2} \right).
\end{align}
We can rewrite the difference scheme above the same as \eqref{lu_f} in point-to-point form
\begin{equation}
\mathcal{L}^n_jH^n_j=G^n_j,\label{discrete_lu_f}
\end{equation}
where $H^n_j=H(\phi^n_j, \psi^n_j)$ and $G^n_j$ represents the force term. Notice that $\mathcal{L}_j^n$ is the same as $\mathcal{L}$ and also has two components 

\begin{align}
    \mathcal{L}_j^n = \begin{pmatrix}
        \mathcal{\bar{L}}_j^n\\
        \mathcal{\bar{\bar{L}}}_j^n
    \end{pmatrix},
\end{align}
where
\begin{align*}
        \mathcal{\bar L}^n_jH^n_j &:= \frac{\rho}{\Delta t}\tilde{D}^{2}_1 \left ( \phi_i^{n+1} - \phi_i^n \right )  + \frac{c}{2}\tilde{D}^{2}_1\left ( \phi_i^{n+1} + \phi_i^n \right )  + \frac{EI}{2\Delta x^2}\left ( \delta^2_1\psi^{n+1}_i + \delta^2_1\psi^{n}_i \right ), \\
        \bar{G}^n_j& := \frac{1}{2}\tilde{D}^{2}_1\left ( f^{n+1}_i + f^n_i \right ) .
\end{align*}
and
\begin{align*}
        \mathcal{\bar{ \bar{L}}}^n_jH^n_j &:= \frac{1}{\Delta t}\tilde{D}^{2}_1\left ( \psi_i^{n+1} - \psi_i^n \right )  - \frac{1}{2}\tilde{D}^{2}_1\left ( \phi_i^{n+1} + \phi_i^n \right ) , \\
        \bar{\bar{G}}^n_j& := 0 
\end{align*}
which are the point-by-point form corresponding to \eqref{The final Scheme}.

Since matrix $\mathcal{A}$ is invertible, the equation \eqref{The final Scheme} can be rewritten as
\begin{equation}
    \label{cn_scheme}
    \left( \bm I - \frac{\Delta t}{2}\mathcal{A}^{-1}\mathcal{B} \right)\bm{U}^{n+1} = \left( \bm I + \frac{\Delta t}{2}\mathcal{A}^{-1}\mathcal{B} \right)\bm{U}^n + \frac{\Delta t}{2}\mathcal{A}^{-1}\left(\bm{F}^{n+1}+\bm{F}^{n}\right) .
\end{equation}

Solving the above equation yields vector $\bm U ^n$, noting that vector $\bm \Psi^n \in \bm U ^n$ is the second-order spatial derivative of the original solution from \eqref{variable_assume}, and recall the compact difference scheme with $c=2$

\begin{align}
    D_1^2 u^{n}_i &= \frac{1}{\Delta x^2}\delta^2_1 u^{n}_i,\\
    \label{eq:tanslated U to u}
\end{align}
again, rewrite it as a linear system using the matrix $\bm A$ and $\bm B$
\begin{align}
    \bm A \bm \Psi ^n = -\bm B  {\bm C} ^n.
\end{align}
where ${\bm C} ^n := \left( u^{n}_1,...,u^{n}_{N_x-1} \right)^T$.

\section{Consistency, Stability and Convergence}

In this section, we give analytical proof of the consistency, stability and convergence of the difference scheme  \eqref{The final Scheme}, thereby demonstrating that the difference scheme has excellent properties.

\subsection{Consistency Analysis}
Before proving the consistency of the final scheme and the original equation, we shall prove the consistency of \eqref{discrete_lu_f} first.
\begin{lemma}
The compact scheme \eqref{discrete_lu_f} is pointwise consistent with the partial differential equation \eqref{lu_f} at point $(x,t)$ and 
\begin{equation}
   \left\| \left(\mathcal{L}H-F\right)^n_j-\left(\mathcal{L}^n_j H^n_j-G^n_j\right)\right\|_{\infty}=O(\Delta x^4)+O(\Delta t^2).
\end{equation}
\label{le:consistency of middle shceme}
\begin{proof}
Note that the scheme \eqref{discrete_lu_f}  has two part schemes so we first  expand $\mathcal{\bar L}^n_jH^n_j - \bar{G}^n_j$ into Taylor series at point $(x, t)$,  we have
\begin{align}
    \mathcal{\bar L}^n_jH^n_j - \bar{G}^n_j &=c\phi_j^{n+1}+\rho \left(\frac{\partial \phi}{\partial t}\right)_j^{n+1} +EI \left(\frac{\partial^2 \psi}{\partial x^2}\right)_j^{n+1}-f_j^{n+1} \\
&-\frac{\Delta t}{2}\left [ c\left(\frac{\partial \phi}{\partial t}\right)_j^{n+1}+\rho\left(\frac{\partial^2 \phi}{\partial t^2}\right)_j^{n+1} +EI\left(\frac{\partial^3 \psi}{\partial t\partial x^2}\right)_j^{n+1}-\left(\frac{\partial f}{\partial t}\right)_j^{n+1}  \right ] \\
&+\frac{(\Delta x)^2}{12}\left[c\left(\frac{\partial^2 \phi}{\partial x^2}\right)_j^{n+1} +\rho\left(\frac{\partial^3 \phi}{\partial t\partial x^2}\right)_j^{n+1}+EI\left(\frac{\partial^4 \psi}{\partial x^4}\right)_j^{n+1}-\left(\frac{\partial^2 f}{\partial x^2}\right)_j^{n+1}\right]\\
&-\frac{\Delta t(\Delta x)^2}{24}\left[c\left(\frac{\partial^3 \phi}{\partial t\partial x^2}\right)_j^{n+1}+\rho\left(\frac{\partial^4 \phi}{\partial t^2\partial x^2}\right)_j^{n+1} +EI\left(\frac{\partial^5 \psi}{\partial t\partial x^4}\right)_j^{n+1}-\left(\frac{\partial^3 f}{\partial t\partial x^2}\right)_j^{n+1}\right]\\
&+O(\Delta t^2)+O(\Delta x^4),
\end{align}
and then use the  equation \eqref{eq1} to eliminate the relevant terms to obtain
\begin{equation}
   \mathcal{\bar L}^n_jH^n_j - \bar{G}^n_j = O(\Delta x^4)+O(\Delta t^2).\label{consistence1}
\end{equation}
We can get the following result by similar derivation to the second scheme 
\begin{equation}
       \mathcal{\bar {\bar{L}}}^n_jH^n_j- \bar{\bar{G}}_j^n = O(\Delta x^4)+O(\Delta t^2).\label{consistence2}
\end{equation}
Combining \eqref{consistence1} and \eqref{consistence2} leads to the consistence:
\begin{align}
    \left\| \left(\mathcal{L}H-F\right)^n_j-\left(\mathcal{L}^n_j H^n_j-G^n_j\right)\right\|_{\infty} 
    &=\left\|\begin{pmatrix}
        (\mathcal{ \bar L}H-\bar F)^n_j-(\mathcal{ \bar L}^n_j H^n_j-\bar G^n_j)\\
        (\mathcal{ \bar {\bar L}}H - \bar{\bar{G}})^n_j-(\mathcal{ \bar{\bar L}}^n_j H^n_j - \bar{\bar{G}}_j^n)
    \end{pmatrix}\right\|_{\infty}\\
    &=    \max \{ (\mathcal{ \bar L}H-\bar F)^n_j-(\mathcal{ \bar L}^n_j H^n_j-\bar G^n_j),(\mathcal{ \bar {\bar L}}H - \bar{\bar{G}})^n_j-(\mathcal{ \bar{\bar L}}^n_j H^n_j - \bar{\bar{G}}_j^n)\}\\
    &=\mathcal{O}(\Delta x^4)+\mathcal{O}(\Delta t^2).
\end{align}

\end{proof}
It shows the scheme \eqref{The final Scheme} is second-order accurate in time and fourth-order accurate in space.
\end{lemma}

\begin{theorem}[Consistency]
The final scheme is pointwise consistent with the origin equation \eqref{ebe} and is second-order accurate in time and fourth-order accurate in space.
\begin{proof}
    It is natural to deduce that scheme \eqref{eq:tanslated U to u} is consistent with the equation
    \begin{align}
    u_{xx} (x,t)= \psi(x,t).
    \end{align}
    since \eqref{D23} has already ensured it and is fourth-order accurate in space. From lemma \ref{le:consistency of middle shceme} we know that scheme \eqref{discrete_lu_f} is consistent with equations \eqref{eq1}, which is actually transformed from equation \eqref{ebe}, so we must have the final scheme where we get numerical solution $u(x_i,t_n)$ is pointwise consistent with the origin equation \eqref{ebe},  and is second-order accurate in time and fourth-order accurate in space.
\end{proof}
    
\end{theorem}

\subsection{Stability Analysis}

Since the right end term matrix of \eqref{cn_scheme} has no effect on the stability of the scheme, we will simplify the equation to the following form when analyzing its stability
\begin{equation}
    \mathcal{Q}_1\bm U ^{n+1} =\mathcal{Q}_2\bm U ^n.
\end{equation}
here we denote
\begin{align}
    \mathcal{Q}_1 &= \left( \bm I - \frac{\Delta t}{2}\mathcal{A}^{-1}\mathcal{B} \right),\label{Q1}\\
    \mathcal{Q}_2 &= \left( \bm I + \frac{\Delta t}{2}\mathcal{A}^{-1}\mathcal{B} \right)\label{Q2},
\end{align}
and the amplification matrix $\mathcal{Q}$ is given by
\begin{equation}
    \mathcal{Q} = \mathcal{Q}_1^{-1} \mathcal{Q}_2.
    \label{Q}
\end{equation}
In order to obtain the stability of the scheme proposed in the previous section, we need to verify whether the spectral radius of the amplification matrix satisfies the von Neumann stability condition, that is, to verify
\begin{equation}
    \left|\rho(\mathcal{Q})\right| = \sup_{\mathcal{Q}\bm x=\lambda \bm x,\bm x\neq \bm 0} \left| \lambda \right| \leq 1.
    \label{eq:Spectral radius of Spectral radius leqslant 1}
\end{equation}
where $\rho(\mathcal{Q})$ represents the spectral radius of $\mathcal{Q}$. 

We now give a general statement and proof of the above discussion. First we will introduce some useful lemmas
\begin{lemma}
    All of the eigenvalues of a real symmetric and strictly diagonally dominant matrix with positive diagonal elements are positive real numbers.
\end{lemma}
\begin{proof}
    Let $\bm \Lambda$ be a real symmetric and strictly diagonally dominant matrix with $n$ rows and $n$ columns, and let $\eta$ be the eigenvalue of $\bm\Lambda$ with eigenvector $\bm x$ i.e. $\bm \Lambda \bm x = \eta \bm x$.\par
    For $n=1,2$ the result is trivial.\par
    Here we consider the situation with $n \geq 3$. First $\eta$ must be a real number since $\bm \Lambda$ is real symmetric. Using the Gershgorin Circle Theorem\cite{weisstein2003gershgorin}, we can deduce the estimate of $\eta$, and notice that $\bm \Lambda$ is strictly diagonally dominant i.e. 
    \begin{equation}
   \sum_{j=1,j\neq i}^{n}\left| \Lambda_{ij} \right| < \left| \Lambda_{ii}\right|  ,   
    \end{equation}
    so
    \begin{align}
        \label{eq_lemma1}
        \left| \Lambda_{ii} - \eta\right| \leq \sum_{j=1,j\neq i}^{n}\left| \Lambda_{ij} \right| < \left| \Lambda_{ii}\right|,
    \end{align}
here $\Lambda_{ij}$ denote the element of $i$th row and $j$th column. since $\Lambda_{ii} > 0$, \eqref{eq_lemma1} holds if only if $\eta > 0$ holds.
\end{proof}
\begin{lemma}
    The relation \eqref{eq:Spectral radius of Spectral radius leqslant 1} holds if and only if  $\mathfrak{Re}(z) \leq0$, where $z$ is the eigenvalue of  matrix $\mathcal{A}^{-1}\mathcal{B}$.
\begin{proof}
    First, we let $\lambda$, $\lambda_1$, $\lambda_2$, and $z$ be eigenvalues of $\mathcal{Q}$, $\mathcal{Q}_1$, $\mathcal{Q}_2$, and $\mathcal{A}^{-1}\mathcal{B}$, respectively. From the definition of $\mathcal{Q}$ implies
    \begin{align}
        \lambda = \frac{\lambda_2}{\lambda_1},
    \end{align}
since $\left(\hat{\bm A} \pm k\bm I\right)\bm x = \left(\hat{\lambda} \pm k\right)\bm x$, for any complex number $k$, and $\hat{\bm A}\bm{x} = \hat{\lambda}\bm{x}$. We deduce
    \begin{align}
        \lambda_1 &= 1-\frac{\Delta t}{2}z ,\\
        \lambda_2 &= 1+\frac{\Delta t}{2}z ,
    \end{align}
combine the results above implies
    \begin{equation}
        \lambda = \frac{1+\frac{\Delta t}{2}z}{1-\frac{\Delta t}{2}z} .
    \end{equation}
Let's set $z\Delta t = a + bi$ where  $a,b\in \mathbb{R}$. A simple computation yields
\begin{equation}
    \left | \lambda_1 \right | ^2 = (1 + \frac{a}{2} )^2 + \frac{b^2}{4} ,
\end{equation}
and
\begin{equation}
    \left | \lambda_2 \right | ^2 = (1 - \frac{a}{2} )^2 + \frac{b^2}{4} .
\end{equation}
From these calculations, we observe that $\left | \lambda_1 \right | \leq \left | \lambda_2 \right |$ if and only if $a \leq 0$and $\left | \lambda_1 \right | > \left | \lambda_2 \right |$ if and only if $a > 0$. Therefore, the relation \eqref{eq:Spectral radius of Spectral radius leqslant 1} holds if and only if $z$ be in the left-half complex plane, that is $\mathfrak{Re}(z) \leq0$.

\end{proof}
\end{lemma}

    Now we only need to prove that all the eigenvalues of $\mathcal{A}^{-1}\mathcal{B}$ have a negative real part. Before proof it, we should proof a lemma.
\begin{lemma}The matrices $ A = diag_n\left(b,a,b\right)$ and $B = diag_n\left(d,c,d\right)$ are commutative.
\label{le:commutative of three diag matricx}
    \begin{proof}
        We proceed by induction on the size of the matrices $A$ and $B$.
    
        First, for $n=1$, the conclusion is trivial.
    
        For $n=2$, through direct computation, we have
        \begin{equation}
            AB = \begin{pmatrix} a & b \\ b & a \end{pmatrix}
            \begin{pmatrix} c & d \\ d & c \end{pmatrix}
            = \begin{pmatrix} ac + bd & ad + bc \\ bc + ad & bd + ac \end{pmatrix} 
            = \begin{pmatrix} c & d \\ d & c \end{pmatrix}
            \begin{pmatrix} a & b \\ b & a \end{pmatrix}
            = BA.
        \end{equation}
    
        Assume that the conclusion holds for $n-1$, for the case of $n$,
        \begin{equation}
            A = \begin{pmatrix} A_1 & \alpha \\ \alpha^T & a \end{pmatrix}, 
            B = \begin{pmatrix} B_1 & \beta \\ \beta^T & c \end{pmatrix},
        \end{equation}
        where $\alpha = (0, \dots, b)^T \in \mathbb{C}^{n-1}$, $\beta = (0, \dots, d)^T \in \mathbb{C}^{n-1}$, $A_1 = \operatorname{diag}\{b, a, b\}_{n-1}$, and $B_1 = \operatorname{diag}\{d, c, d\}_{n-1}$. Computing the products, we obtain
        \begin{equation}
            AB = \begin{pmatrix} A_1 & \alpha \\ \alpha^T & a \end{pmatrix}
            \begin{pmatrix} B_1 & \beta \\ \beta^T & c \end{pmatrix}
            = \begin{pmatrix} A_1B_1 + \alpha\beta^T & A_1\beta + c\alpha \\ \alpha^TB_1 + a\alpha^T & \alpha^T\beta + ac \end{pmatrix},
        \end{equation}
        and
        \begin{equation}
            BA = \begin{pmatrix} B_1 & \beta \\ \beta^T & c \end{pmatrix}
            \begin{pmatrix} A_1 & \alpha \\ \alpha^T & a \end{pmatrix}
            = \begin{pmatrix} B_1A_1 + \beta\alpha^T & B_1\alpha + a\beta \\ \beta^TA_1 + c\alpha^T & \beta^T\alpha + ca \end{pmatrix}.
        \end{equation}
    
        Since both $\alpha$ and $\beta$ are vectors, it follows that $\alpha\beta^T = \beta\alpha^T$ and $\alpha^T\beta = \beta^T\alpha$. By the induction hypothesis, we know that $A_1B_1 = B_1A_1$. Thus, we have
        \begin{align}
            A_1B_1 + \alpha\beta^T &= B_1A_1 + \beta\alpha^T, \\
            \alpha^T\beta + ac &= \beta^T\alpha + ca.
        \end{align}
    
        Moreover, direct computation yields
        \begin{align}
            A_1\beta + c\alpha &= B_1\alpha + a\beta = (0, \dots, bd, ad + bc)^T, \\
            \alpha^TB_1 + a\alpha^T &= \beta^TA_1 + c\alpha^T = (0, \dots, bd, ad + bc).
        \end{align}
    
        Therefore, it follows that $AB = BA$. By the principle of induction, the lemma is proved.
    \end{proof}
    
\end{lemma}

\begin{remark}
    \label{remark:the commmutative of A^{-1} and B}
    From the above lemma, we know that the matrices $\bm A$ and $\bm B$, as defined in the finite difference method, are commutative, i.e., $\bm A \bm B = \bm B \bm A$. Since $\bm A$ is invertible, we multiply both sides of the equation $\bm A \bm B = \bm B \bm A$ on the left by $\bm A^{-1}$ and on the right by $\bm A^{-1}$, yielding $\bm A^{-1} \bm B = \bm B \bm A^{-1}$.
\end{remark}

\begin{lemma}
If $EI, \rho $ and $c$ are all positive numbers, then all the eigenvalues of $\mathcal{A}^{-1}\mathcal{B}$ have negative real part.
\end{lemma}
\begin{proof}
        Let $\eta = x + iy$ be an eigenvalue of the matrix $\mathcal{C} := \mathcal{A}^{-1}\mathcal{B}$ and $\zeta = \begin{pmatrix}
            \zeta_1\\ \zeta_2
        \end{pmatrix}$ be the corresponding non-zero eigenvector, where $\zeta \in \mathbb{C}^{2(N_x-1)},\zeta_1,\zeta_2\in \mathbb{C}^{N_x-1}$, i.e.
			\begin{equation}
             \mathcal{C}\zeta = 
				\mathcal{A}^{-1}\mathcal{B}\zeta = 
                \begin{pmatrix}
			      \frac{1}{\rho}\bm A^{-1}& \bm 0  \\  \bm 0& \bm A^{-1}    \end{pmatrix} 
			    \begin{pmatrix}
			      -c\bm A& EI \bm B  \\  -\bm B & \bm 0   \end{pmatrix} 
			    \mathbb{\zeta} = 
                \begin{pmatrix}
			      -\frac{c}{\rho}\bm I& \frac{EI}{\rho}\bm A^{-1} \bm B \\  -\bm A^{-1} \bm B & \mathbf{0}    \end{pmatrix} 
			    \mathbb{\zeta} =    
			    \eta \mathbb{\zeta}.
			    \label{eq:proof_C_long}
			\end{equation}
So we have two equation of  $\zeta_1$ and $\zeta_2$
\begin{equation*}
		    \left\{
		    \begin{aligned}
		    -\frac{c}{\rho}\zeta_1 + \frac{EI}{\rho}\bm{A}^{-1}\bm{B}\zeta_2 &= \eta \zeta_1,\\
		    -\bm{A}^{-1}\bm{B}\zeta_1  & = \eta \zeta_2.\\
		    \end{aligned}
		    \right.
		\end{equation*}
Then the relation between $\zeta_1$ and $\zeta_2$ is 
\begin{equation}
 \bm{A^{-1}B}\zeta_2 = \frac{\rho \eta + c}{EI}\zeta_1.
 \label{eq:the relation of zeta1 and zeta2}
\end{equation}
        Take the conjugate transpose on both sides of the equation \eqref{eq:proof_C_long} gives $\zeta^*\mathcal{C}^*  = \overline{\eta}\zeta^*$. Right multiply vector $\zeta$, we have 
			\begin{equation}
			    \zeta^*\mathcal{C}^*\zeta  = \overline{\eta}\zeta^*\zeta.
			    \label{eq:proof_B_1}
			\end{equation}
		Left multiply equation (\ref{eq:proof_C_long}) with vector $\zeta^*$, we have
			\begin{equation}
				\zeta^*\mathcal{C}\zeta  = 
				\eta\zeta^*\zeta.
				\label{eq:proof_B_2}
			\end{equation}		
	Note that $\bm A^{-1} ,\bm B$ are both symmetric matrices.We also know $\bm A^{-1} \bm B = \bm B \bm A^{-1}$from remark \ref{remark:the commmutative of A^{-1} and B}, so we have
		\begin{equation}
		    \mathcal{C}^* = 
		     \begin{pmatrix} -\frac{c}{\rho}\bm I
             & -\bm A^{-1} \bm B  \\  
		     \frac{EI}{\rho}\bm A^{-1} \bm B  & \mathbf{0}.		     			 
		     \end{pmatrix} 
		\end{equation}
		So add equation (\ref{eq:proof_B_1}) and (\ref{eq:proof_B_2}) side by side,  and  use  the relation \eqref{eq:the relation of zeta1 and zeta2} , we can perform the following calculation
		\begin{align*}
	    \zeta^*(\bm C^* + \bm C)\zeta  &=
	    \begin{pmatrix}
	        \zeta_1^* & \zeta_2^*
	    \end{pmatrix}
	    \begin{pmatrix}
	        -\frac{2c}{\rho}\bm I & (\frac{EI}{\rho} - 1)\bm A^{-1} \bm B  \\ 
	        (\frac{EI}{\rho} - 1)\bm A^{-1} \bm B & \bm 0
	    \end{pmatrix}
	    \begin{pmatrix}	        \zeta_1 \\	        \zeta_2	    \end{pmatrix} \\ 
	    &= -\frac{2c}{\rho}\zeta_1^*\zeta_1 
	    + \left(\frac{EI}{\rho} - 1\right)\zeta_1^*\bm A^{-1} \bm B\zeta_2  
	    + \left(\frac{EI}{\rho} - 1\right)\zeta_2^*\bm A^{-1} \bm B\zeta_1\\
	    &= -\frac{2c}{\rho}\zeta_1^* \zeta_1 
	    + 2\mathfrak{Re} \left[\left(\frac{EI}{\rho} - 1\right)\zeta_1^*\bm A^{-1} \bm B\zeta_2\right]\\
	    &= -\frac{2c}{\rho}\zeta_1^*\zeta_1 
	    + 2\left(\frac{EI}{\rho} - 1\right)\mathfrak{Re} \left[\zeta_1^*\bm A^{-1} \bm B\zeta_2\right]\\
	    &= -\frac{2}{\rho} \left\{c\zeta_1^*\zeta_1 
	    + (\rho - EI)\mathfrak{Re} \left[\zeta_1^*\bm A^{-1} \bm B\zeta_2\right]\right\}\\
	    &= -\frac{2}{\rho} \left[c +\left (\frac{\rho}{EI} - 1\right)\mathfrak{Re}(\rho\eta+c)\right]\zeta_1^*\zeta_1\\
	    &= -2\left[\left(\frac{\rho}{EI} - 1\right)\mathfrak{Re}(\eta) + \frac{c}{EI}\right]\zeta_1^*\zeta_1\\
        &= (\eta + \overline{\eta})\zeta^*\zeta \\
	    &= 2\mathfrak{Re}(\eta)\zeta^*\zeta ,
	\end{align*}
	Taking the last formula and the third to last formula of calculation above, we have
    \begin{equation*}
        \left[\left(1 - \frac{\rho}{EI}\right)\mathfrak{Re}(\eta) - \frac{c}{EI}\right]\zeta_1^*\zeta_1 = \mathfrak{Re}(\eta)\zeta^*\zeta ,
    \end{equation*}

    Note that  $\zeta^*\zeta = \zeta_1^*\zeta_1 + \zeta_2^*\zeta_2$, then  we have
    \begin{equation*}
        -\left(\frac{\rho}{EI}\zeta^*_1\zeta_1 + \zeta^*_2\zeta_2\right) \mathfrak{Re}(\eta) =         \frac{c}{EI}\zeta_1^*\zeta_1.
    \end{equation*}
     It's easy to note that $\zeta_1^*\zeta_1 \ge 0$ and $\zeta^*_2\zeta_2 \ge 0$. And by assumption, so we must have $\mathfrak{Re}(\eta) \le 0$.
\end{proof}

\begin{remark}
    The conditions in the theorem are that $EI$,$\rho$,$c$ are all positive numbers, and in the actual physical background, these values are all positive numbers. Therefore, we have proved the stability of this method in general engineering calculations.
\end{remark}

\begin{remark}
    Note that if only the real part of the eigenvalue of $\mathcal{B}$ is less than or equal to 0, it cannot be obtained that the real part of the eigenvalue of $\mathcal{A}^{-1}\mathcal{B}$ is less than or equal to 0, where $\mathcal{A}$ is defined by \eqref{eq:the definition of matrix A and B}. For example, let $EI = \rho = c = 1$, and $\mathcal{A},\mathcal{B} $ are both  fourth-order matrices, where 
    \begin{equation*}
        \mathcal{B} = \begin{bmatrix}
        -3 & -4 & 3 & -2 \\
        5 & 0 & 0 & 2 \\
        4 & -1 & -1 & 5 \\
        -2 & 5 & -4 & -2
        \end{bmatrix}
    \end{equation*}
then we can compute that the eigenvalues of  $\mathcal{B}$  and $\mathcal{A}^{-1}\mathcal{B}$ are
\begin{table}[h!]
\centering
\begin{tabular}{|c|c|}
\hline
\textbf{Eigenvalues of $\mathcal{B}$}& \textbf{Eigenvalues of $\mathcal{A}^{-1}\mathcal{B}$}\\ \hline
$-4.2503 + 0.0000i$ & $-6.4027 + 0.0000i$ \\ \hline
$-1.3154 + 0.0000i$ & $-1.2627 + 0.0000i$ \\ \hline
$-0.2172 + 4.1164i$ & $0.0751 + 4.9855i$ \\ \hline
$-0.2172 - 4.1164i$ & $0.0751 - 4.9855i$ \\ \hline
\end{tabular}
\caption{Eigenvalues of $\mathcal{B}$ and $\mathcal{A}^{-1}\mathcal{B}$}
\label{tab:eigenvalues}
\end{table}
It can be found that the real part of the eigenvalues of $\mathcal{B}$ is less than or equal to 0, but $\mathcal{A}^{-1}\mathcal{B}$ is not in table (\ref{tab:eigenvalues}). Therefore, the proofs of stability analysis in many literature are completely wrong, such as \cite{aouragh2024compact}.    

\end{remark}
 According to the above discussion, we get the following theorem
 
\begin{theorem}[Stability]
 If  $E,I,\rho,c$ are all positive constants, then the scheme \eqref{The final Scheme} is unconditionally stable.   
 \end{theorem}
 
\subsection{Convergence Analysis}
Now we have proven the stability of the scheme. The convergence of the numerical scheme  can be concluded using the Lax Richtmyer theorem: stability and consistency imply convergence\cite{lax2005survey}. So now we have the convergence theorem below:

\begin{theorem}[Convergence]
    Let \( u(x,t)\in \mathcal{U} \) be the analytical solution of the Euler-Bernoulli beam equation, 
and assume \( u_{xxxx} \) and \( u_{tt} \) are bounded in the domain. The compact difference scheme is convergence with infinity-norm, further, 
\begin{equation}
    \| \mathbf u ^n-\mathbf v^n\|_{l^{\infty}}=O(\Delta x^4)+O(\Delta t^2), n\in \mathbf T,
\end{equation}
where $\mathbf u^n$ and $\mathbf{ v}^n$ represent the solution vector of  numerical solution and exact solution at $n$-th time step, respectively.
\end{theorem}

\section{Numerical Experiments}
In this section, we present the numerical results of the proposed method applied to three examples of beam equations existing in the literature. We note that all calculations are achieved using the JULIA language on a MacBook Pro computer equipped with the processor 1,8 GHz Intel Core i5. We test the accuracy and stability of the proposed method by using it with different values of  $\Delta t$ and $h$. The $L^{\infty}$ error obtained for the method is shown. We also calculate the computational orders of our method (denoted by \textit{C-Order}) using the following formula
\subsection{Example 1}
\begin{example}
    Set $EI = 98,\rho = 0.685,c =0.75$ in model \eqref{ebe} to \eqref{ebe-boundary-condition} and the force term $f$ is chosen such that exact solution is 
     \begin{equation}
        u(x,t) = \sin(\pi x)\cos(\pi t).
    \end{equation}
\end{example}

The exact solution ,numerical solution and error between exact and numerical solution, are plotted on 2\textit{D }for $h=0.01$ and  $t = 0.005$ in fig \ref{fig:example1}. The maximum absolute errors in the solution $u$ and convergence in time of example 1, are listed in Table \ref{tab:x error  order of example 1} for different values of $h$ at $t=1$. Fig \ref{fig:x order of example 1} and \ref{fig:t order of example 1} show the convergence order of Example 1 in space and time, respectively.
\begin{figure}[htbp]
    \centering
    \includegraphics[width=1\linewidth]{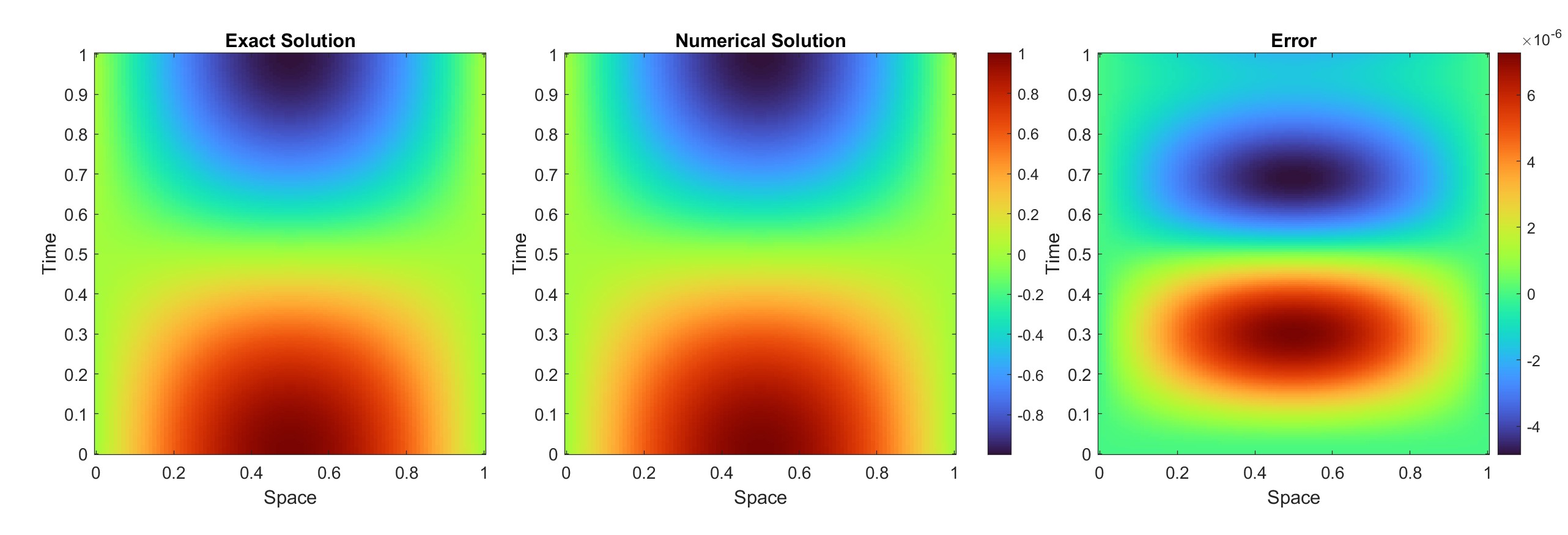}
    \caption{Example 1: Exact solution (left), Numerical solution (middle) and Error (right) for for $h = 0.01$ and $\Delta t = 0.005$}
    \label{fig:example1}
\end{figure}


\begin{table}[h]
    \centering
    \renewcommand{\arraystretch}{1.2} 
    \begin{tabular}{c c c c c} 
        \hline
        Mesh & $N$ & $h$ & $e_h$ & $C$-order \\ 
        \hline
        $Mesh_1$ & 32  & 0.0312& 6.634501648061786e-7
& -\\
        $Mesh_2$& 64  & 0.0156& 4.1534138461862824e-8
& 3.997618567152459
\\
        $Mesh_3$& 128 & 0.0078& 2.597110193569563e-9
& 3.999318496099914
\\
        $Mesh_4$& 256 & 0.0039& 1.6252299506192003e-10
& 3.9981914660506632
\\
        $Mesh_5$& 512 & 0.0020& 1.0099698855015049e-11& 4.008259675123924\\
        Average  &     &   &    &         4.000847051\\
        \hline
    \end{tabular}
    \caption{Example 1: $L^{\infty}$ errors for different values of $h$ and $\Delta t = h^2$ at $t = 1 $}
    \label{tab:x error  order of example 1}
\end{table}

\begin{figure}[h!]
    \centering
    \begin{minipage}{0.45\textwidth}
        \centering
\label{fig:x order of example 1}
        \includegraphics[width=1\textwidth]{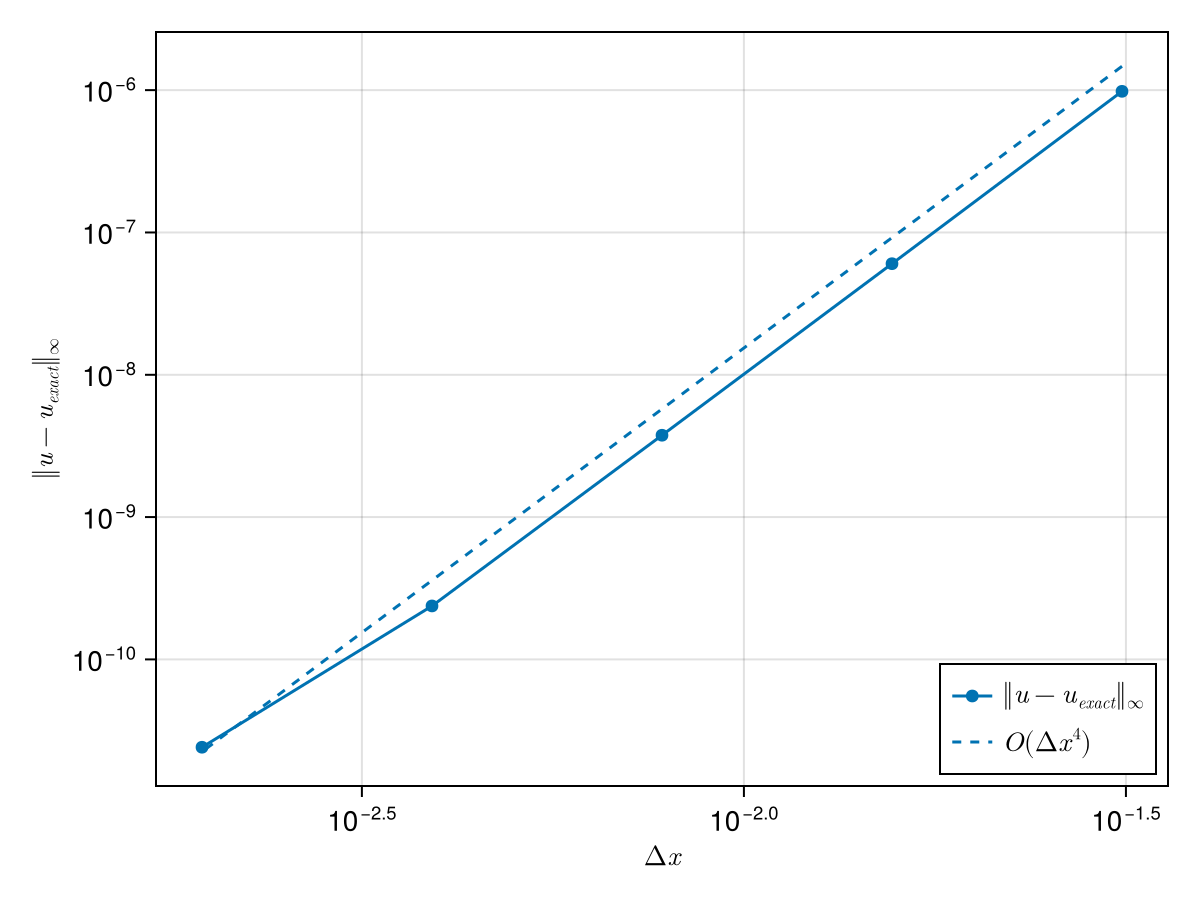}
        \caption{Example 1: Error order with $\Delta t$ of example 1}
    \end{minipage} \hfill
    \begin{minipage}{0.45\textwidth}
        \centering
        \label{fig:t order of example 1}
        \includegraphics[width=1\textwidth]{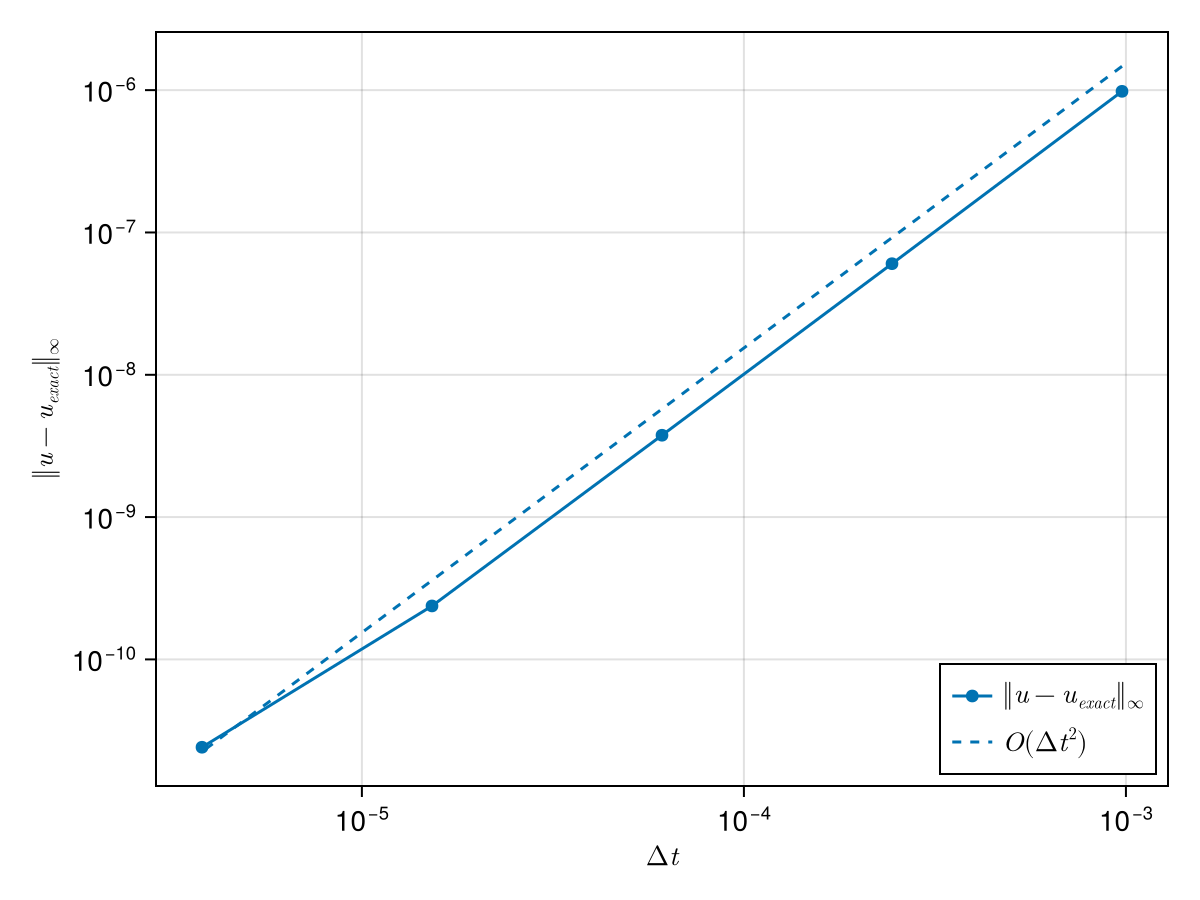}
        \caption{Example 1: Error order with $\Delta t$ of example 1}
    \end{minipage}
\end{figure}

\subsection{Example 2}
\begin{example}
     Set $EI = 1,\rho = 1,c =1$ in model \eqref{ebe} to \eqref{ebe-boundary-condition} and the force term $f$ is chosen such that exact solution is 
    \begin{equation}
    u(x,t) = \sinh(t)\cos(\pi x).
\end{equation}
\end{example}

The exact solution, numerical solution and error between exact and numerical solution, are plotted on 2\textit{D }for $h=0.01$ and  $t = 0.005$ in fig \ref{fig:example2}. The maximum absolute errors in the solution variable $u$ of example 2, are listed in Table \ref{tab:x error  order of example 2} for different values of $h$ at $t=1$. Fig \ref{fig:x order of example 2} and \ref{fig:t order of example 2} show the convergence order of Example 2 in space and time, respectively.
\begin{figure}[htbp]
    \centering
    \includegraphics[width=1\linewidth]{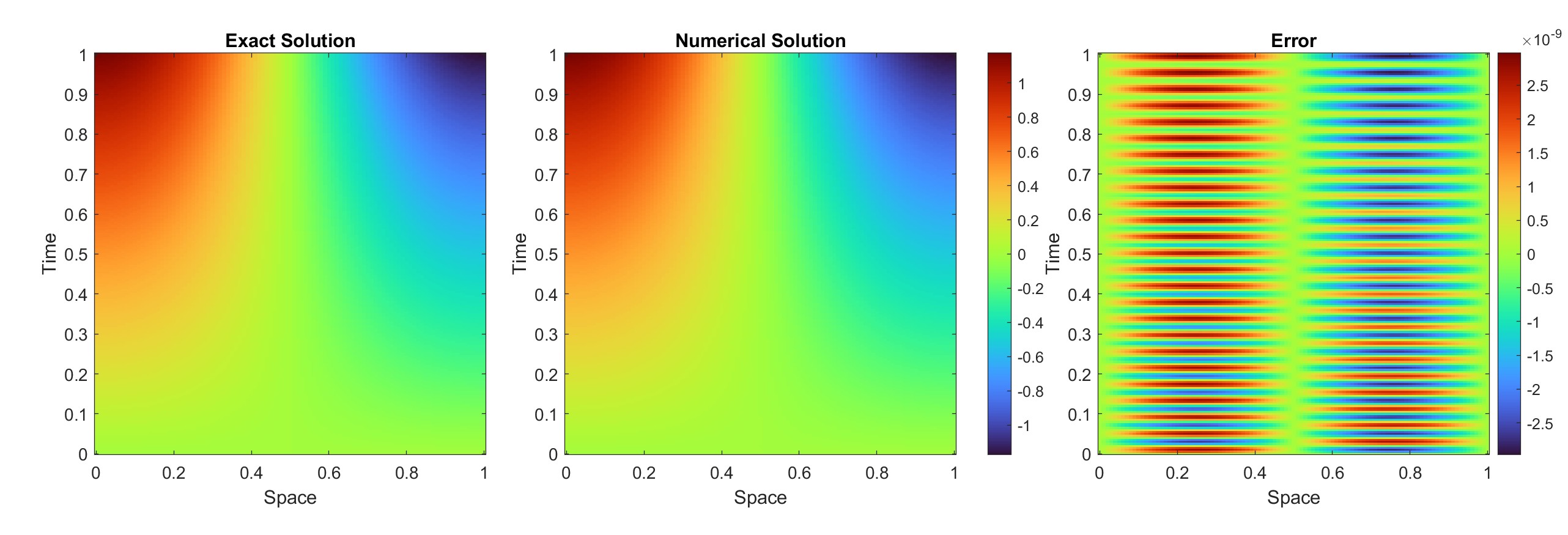}
    \caption{Example 2: Exact solution (left), Numerical solution (middle) and Error (right) for for $h = 0.01$ and $\Delta t = 0.005$}
    \label{fig:example2}
\end{figure}

\begin{table}[h]
    \centering
    \renewcommand{\arraystretch}{1.2} 
    \begin{tabular}{c c c c c} 
        \hline
        Mesh & $N$ & $h$ & $e_h$ & $C$-order \\ 
        \hline
        $Mesh_1$ & 32  & 0.0312& 1.182547729E-07
& -\\
        $Mesh_2$& 64  & 0.0156& 7.388819223E-09
& 4.000410771
\\
        $Mesh_3$& 128 & 0.0078& 4.620274163E-10
& 3.999293464
\\
        $Mesh_4$& 256 & 0.0039& 2.901023866E-11
& 3.993344393
\\
        $Mesh_5$& 512 & 0.0020& 1.859623566E-12
& 3.963479645
\\
        Average  &     &   &    &         3.989132068
\\
        \hline
    \end{tabular}
    \caption{Example 2: $L^{\infty}$ errors for different values of $h$ and $\Delta t = h^2$ at $t = 1 $}
    \label{tab:x error order of example 2}
\end{table}

\begin{figure}[h!]
    \centering
    \begin{minipage}{0.45\textwidth}
        \centering
        \label{fig:t order of example 2}
        \includegraphics[width=1\textwidth]{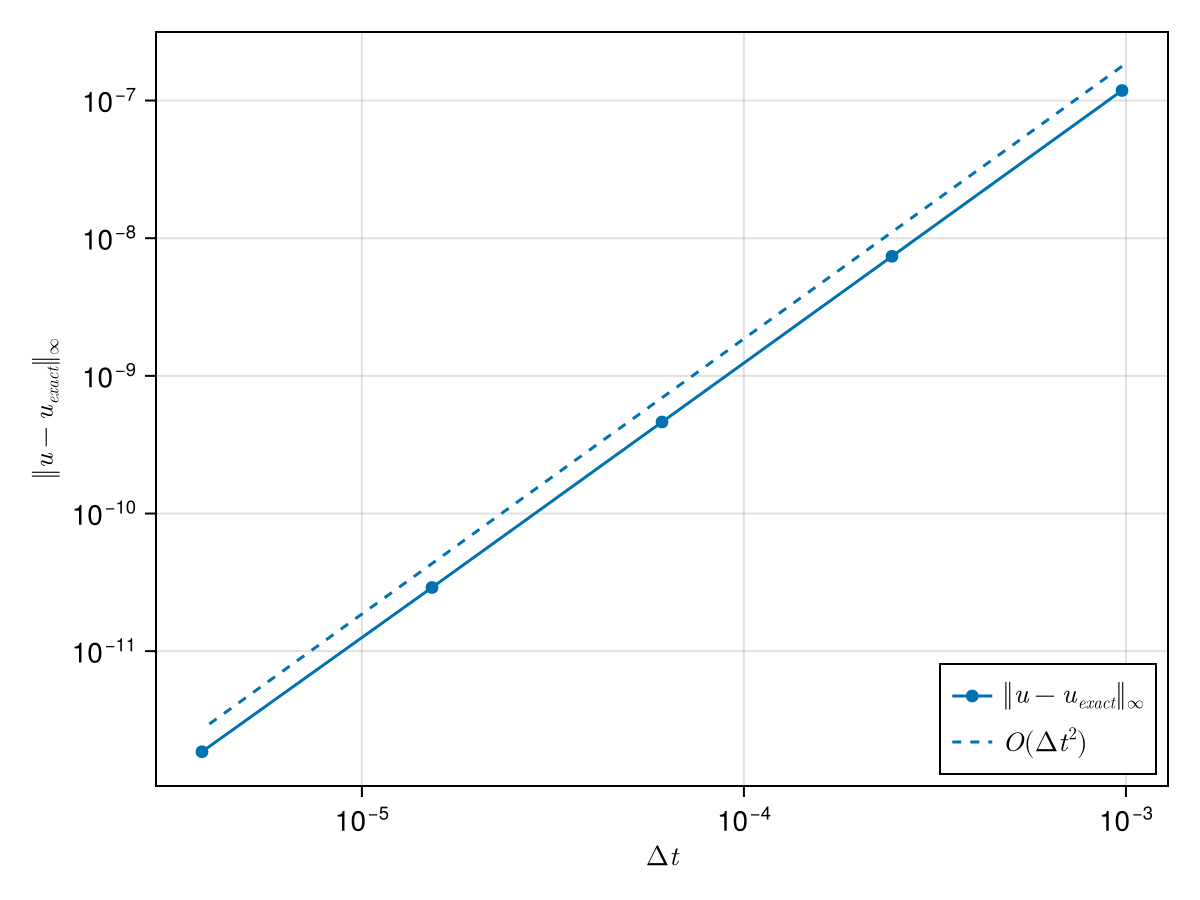}
        \caption{Error order with $\Delta t$ of example 2}
    \end{minipage} \hfill
    \begin{minipage}{0.45\textwidth}
        \centering
        \label{fig:x order of example 2}
        \includegraphics[width=1\textwidth]{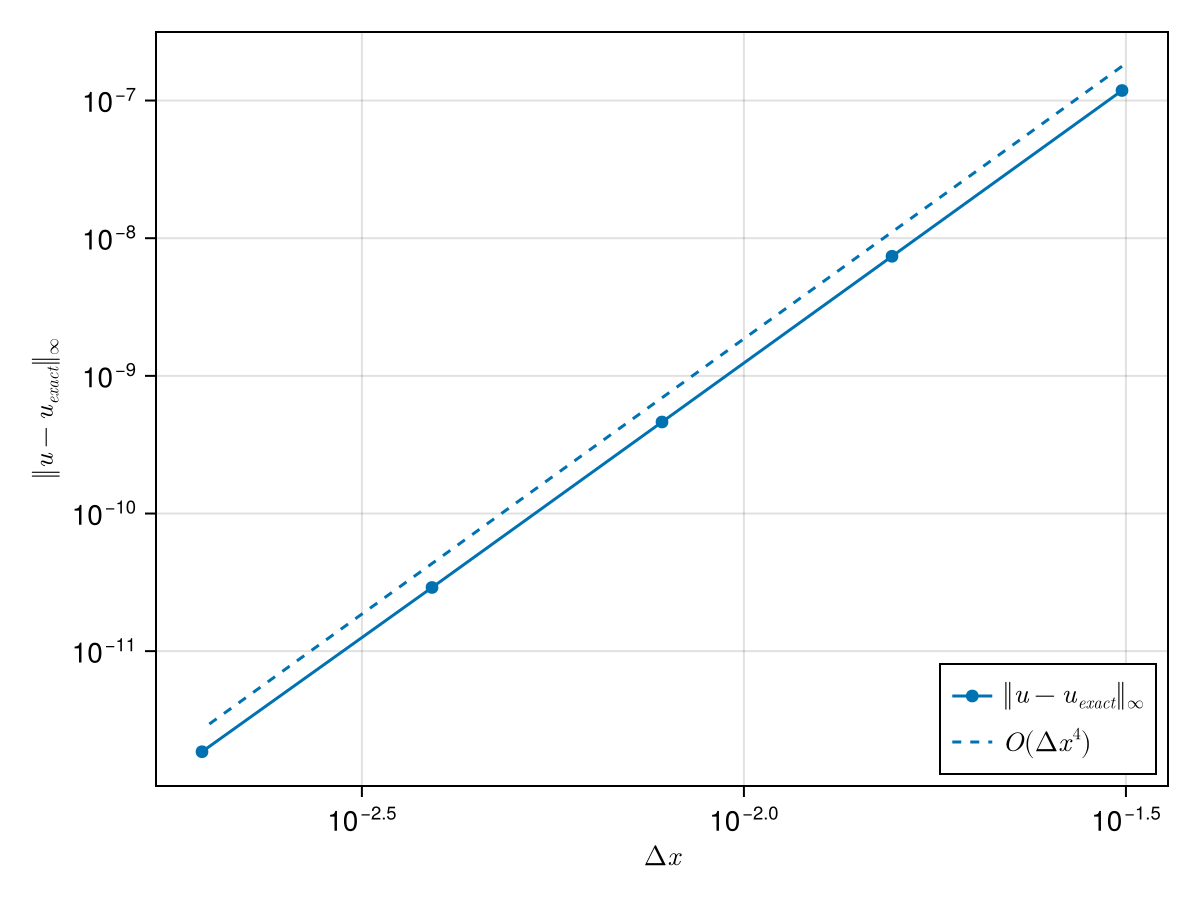}
        \caption{Error order with $\Delta t$ of example 2}
    \end{minipage}
\end{figure}

\subsection{Example 3}

\begin{example}
    Set $EI = 98,\rho = 0.68,c =7.5$ in model \eqref{ebe} to \eqref{ebe-boundary-condition} and the force term $f$ is chosen such that exact solution is 
\begin{equation}
    u(x,t) = e^{-t}\sin(\pi x).
 \end{equation}
\end{example}

The exact solution,numerical solution and error between exact and numerical solution, are plotted on 2\textit{D }for $h=0.01$ and  $t = 0.005$ in fig \ref{fig:example3}. The maximum absolute errors in the solution variable $u$ of example 1 , are listed in Table \ref{tab:x error  order of example 3} for different values of $h$ at $t=1$. Fig \ref{fig:x order of example 3} and \ref{fig:t order of example 3} show the convergence order of Example 1 in space and time, respectively.In the error map, super-convergence occurs when either the time step or the spatial step is small enough, probably because the spatial and temporal steps satisfy a specific ratio making the error order higher.
\begin{figure}[htbp]
    \centering
    \includegraphics[width=1\linewidth]{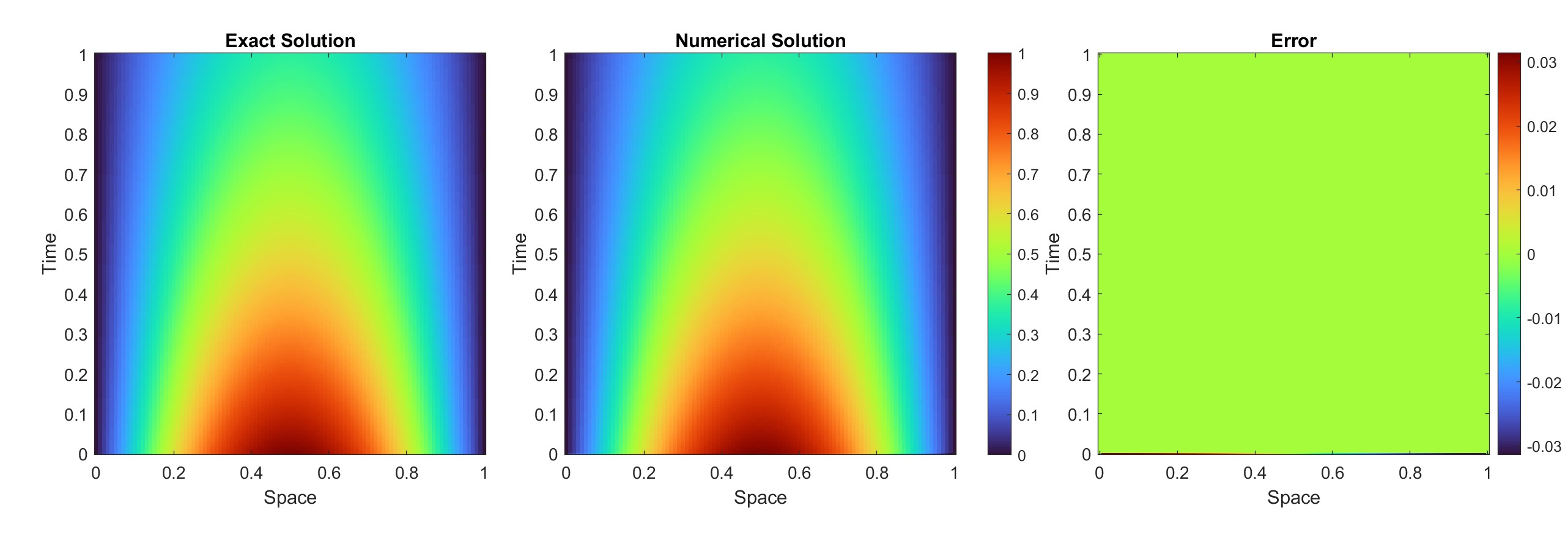}
    \caption{Example 3: Exact solution (left), Numerical solution (middle) and Error (right) for for $h = 0.01$ and $\Delta t = 0.005$}
    \label{fig:example3}
\end{figure}


\begin{table}[htbp]
    \centering
    \renewcommand{\arraystretch}{1.2} 
    \begin{tabular}{c c c c c} 
        \hline
        Mesh & $N$ & $h$ & $e_h$ & $C$-order \\ 
        \hline
        $Mesh_1$ & 32  & 0.0312& 2.849383778924519e-7
& -\\
        $Mesh_2$& 64  & 0.0156& 1.7771973226388127e-8
& 4.002974170882831
\\
        $Mesh_3$& 128 & 0.0078& 1.110404379556229e-9
& 4.000446804731633
\\
        $Mesh_4$& 256 & 0.0039& 6.894679271951532e-11
& 4.009457911233601
\\
        $Mesh_5$& 512 & 0.0020& 1.2934653348395386e-12& 5.736170146624661\\
        Average  &     &   &    &         4.437262258\\
        \hline
    \end{tabular}
    \caption{Example 3: $L^{\infty}$ errors for different values of $h$ and $\Delta t = h^2$ at $t = 1 $}
    \label{tab:x error order of example 3}
\end{table}

\begin{figure}[htbp]
    \centering
    \begin{minipage}{0.45\textwidth}
        \centering
        \label{fig:t order of example 3}
        \includegraphics[width=1\textwidth]{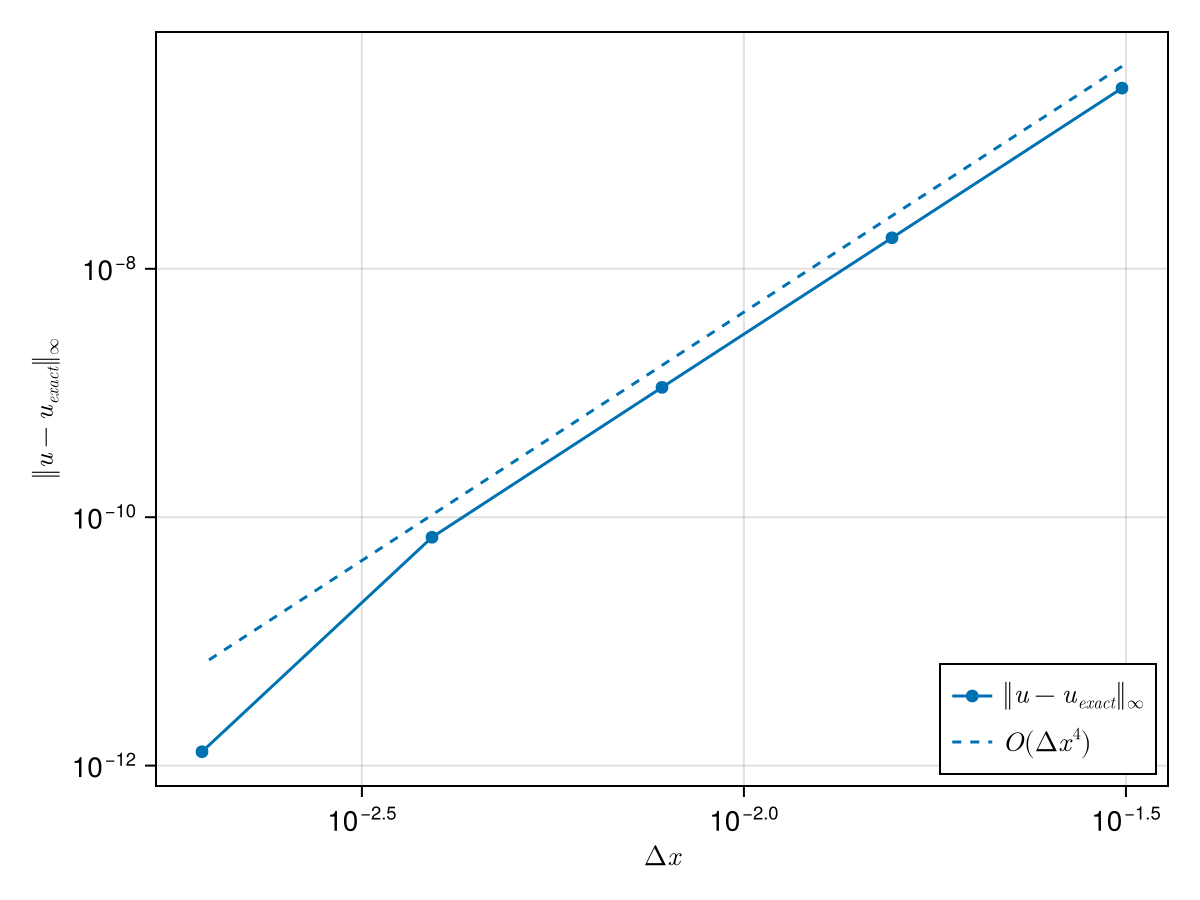}
        \caption{Error order with $\Delta t$ of example 3}
    \end{minipage} \hfill
    \begin{minipage}{0.45\textwidth}
        \centering
        \label{fig:x order of example 3}
        \includegraphics[width=1\textwidth]{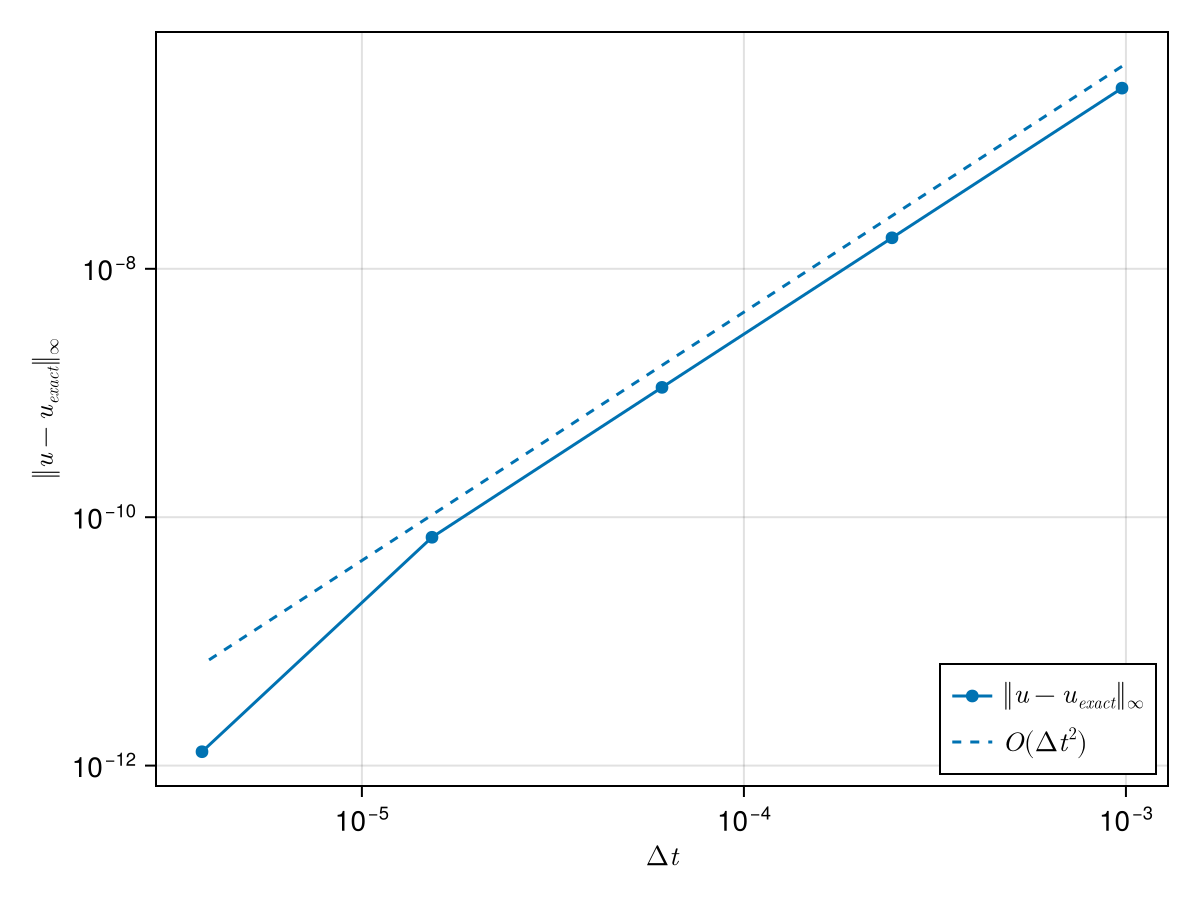}
        \caption{Error order with $\Delta t$ of example 3}
    \end{minipage}
\end{figure}

\section{Conclusion}
In the numerical solution of the classic Euler-Bernoulli beam equation, this paper considers the case of constant coefficients after adding damping terms, and uses the fourth-order compact difference format and the Crank-Nicoslon format finite difference method to solve it, solving the problem that damping terms are usually required in engineering. At the same time, this paper gives the consistency, stability and convergence proof of the compact finite difference method, and conducts a large number of numerical experiments to ensure the feasibility and universality of the method.

\bibliographystyle{unsrt}  
\bibliography{references}  






\end{document}